\documentclass{amsart}

\usepackage[utf8]{inputenc}
\usepackage[T1]{fontenc}
\usepackage{lmodern}
\usepackage[english]{babel}
\usepackage{tikz}
\usepackage{tikz-cd}
\usetikzlibrary{positioning,arrows,shapes}
\usepackage{hyperref}
\usepackage[english,capitalize]{cleveref}
\usepackage{amssymb}
\usepackage{enumitem}
\setenumerate{label=(\roman*)}
\usepackage{combelow}

\theoremstyle{plain}
\newtheorem{prop}{Proposition}[section]
\newtheorem{theorem}{Theorem}
\newtheorem*{theorem*}{Theorem}
\newtheorem{question}{Question}
\newtheorem*{question*}{Question}
\newtheorem*{lemma*}{Lemma}

\newtheorem{cor}[prop]{Corollary}
\newtheorem*{cor*}{Corollary}
\newtheorem{conj}[theorem]{Conjecture}
\newtheorem*{conj*}{Conjecture}

\theoremstyle{definition}
\newtheorem{defi}[prop]{Definition}
\newtheorem*{defi*}{Definition}
\newtheorem{nota}[prop]{Notation}
\newtheorem*{nota*}{Notation}
\newtheorem{rem}[prop]{Remark}
\newtheorem*{rem*}{Remark}
\newtheorem{ex}[prop]{Example}

\newcommand{\RR}{\mathbb{R}}
\newcommand{\ZZ}{\mathbb{Z}}
\newcommand{\NN}{\mathbb{N}}
\newcommand{\CC}{\mathbb{C}}
\newcommand{\QQ}{\mathbb{Q}}
\renewcommand{\AA}{\mathbb{A}}
\renewcommand{\P}{\mathcal{P}}
\newcommand{\Q}{\mathcal{Q}}
\renewcommand{\L}{\mathcal{L}}
\renewcommand{\O}{\mathcal{O}}
\newcommand{\C}{\mathcal{C}}
\newcommand{\F}{\mathcal{F}}

\newcommand{\D}{\mathcal{D}}
\newcommand{\X}{\mathcal{X}}
\newcommand{\EE}{\mathsf{E}_6}
\newcommand{\EEE}{\mathsf{E}_7}
\newcommand{\EEEE}{\mathsf{E}_8}
\newcommand{\FF}{\mathsf{F}_4}
\newcommand{\GG}{\mathsf{G}_2}

\DeclareMathOperator{\conv}{conv}
\let\vert\relax
\DeclareMathOperator{\vert}{vert}
\DeclareMathOperator{\interior}{int}
\let\div\relax
\DeclareMathOperator{\div}{div}
\DeclareMathOperator{\Cl}{Cl}

\title{Gorenstein Fano Toric Degenerations}

\author{Christian Steinert}

\address{Mathematical Institute, Faculty of Mathematics and Natural Sciences, University of Cologne; Chair for Algebra and Representation Theory, RWTH Aachen University}

\email{steinert@art.rwth-aachen.de}

\allowdisplaybreaks

\begin{document}
	
	\begin{abstract}\noindent
		We propose a refined but natural notion of toric degenerations that respect a given embedding and show that within this framework a Gorenstein Fano variety can only be degenerated to a Gorenstein Fano toric variety if it is embedded via its anticanonical embedding. This also gives a precise criterion for reflexive polytopes to appear, which might be required for applications in mirror symmetry. For the proof of this statement we will study polytopes whose polar dual is a lattice polytope. As a byproduct we generalize a connection between the number of lattice points in a rational convex polytope and the Euler characteristic of an associated torus invariant rational Weil divisor, allowing us to show that Ehrhart-Macdonald Reciprocity and Serre Duality are equivalent statements for a broad class of varieties. Additionally, we conjecture a necessary and sufficient condition for the Ehrhart quasi-polynomial of a rational convex polytope to be a polynomial. Finally, we show that the anticanonical line bundle on a Gorenstein Fano variety with at worst rational singularities is uniquely determined by a combinatorial condition of its Hilbert polynomial.
	\end{abstract}
	
	\maketitle

\section{Introduction}

\subsection{Background and Motivation}

It has been a common concept throughout the history of mathematics to translate problems in one area of research to another area of research to solve them there. 

An especially fruitful example of this translating approach has been the study of {\em toric varieties}. These are special varieties whose algebro-geometric properties are completely determined by the combinatorial properties of certain polyhedral objects. This phenomenon has been used by many mathematicians to study more general varieties via flat degenerations to toric varieties. Notable results in this regard have been archived by Gonciulea and Lakshmibai \cite{GL}, Kogan and Miller \cite{KoM}, Caldero \cite{C}, Alexeev and Brion \cite{AB} as well as Feigin, Fourier and Littelmann \cite{FFL3}.

However, there are also constructions that do not originate from representation theory. For example, Okounkov \cite{O,O2}, Lazarsfeld and Musta\cb{t}\u a \cite{LM} as well as Kaveh and Khovanskii \cite{KK} defined and analyzed convex bodies for arbitrary projective varieties\,---\,thereby developing the theory of {\em Newton-Okounkov bodies}.

It has been shown that most of the representation theoretic toric degenerations of flag varieties can be realized via Newton-Okounkov bodies\,---\,for example by Kaveh \cite{Ka}, Kiritchenko \cite{Ki} and Fujita and Naito \cite{FN}. A unified approach has been developed by Fang, Fourier and Littelmann who presented a construction of these polytopes from representation theory via {\em birational sequences} and connected them to Newton-Okounkov Theory in \cite{FaFL}.

Lately, Anderson \cite{A} showed that Newton-Okounkov bodies yield toric degenerations under reasonable technical assumptions, thereby providing a general reason for the existence of the diverse classes of toric degenerations in representation theory mentioned in the beginning. Even more, based on an observation by Alexeev and Brion \cite{AB} he realized that these toric degenerations respect the choice of a given {\em embedding} or {\em polarization} of the original variety.  We will formalize this definition to develop the concept of embedded toric degenerations.

The goal of this paper is to study the following combinatorial property in the context of toric degenerations. A rational convex polytope is called {\em reflexive} if the polytope itself and its polar dual are lattice polytopes (i.e. they have integral vertices). There are two main reasons why these polytopes are important in our context.

Batyrev proved in \cite{B} that reflexive polytopes are in one-to-one correspondence with anticanonically polarized (normal) Gorenstein Fano toric varieties. So essentially we are working towards an answer to the following question.

\begin{question}\label{question:general}
	Which polarized Gorenstein Fano varieties admit a flat projective degeneration to an anticanonically polarized Gorenstein Fano toric variety?
\end{question}

It should be noted that we will only be interested in {\em normal} Gorenstein Fano varieties, so we will leave out the additional adjective.

This poses an interesting problem in itself but there is also another viewpoint towards this framework. In the same work \cite{B}, Batyrev explicitly constructed mirrors to Calabi-Yau hypersurfaces in toric varieties via reflexive polytopes. It deems a reasonable hope that his construction can be generalized to more arbitrary varieties if one were able to associate meaningful reflexive polytopes to those varieties. Of course, (embedded) toric degenerations would prove such a meaningful connection.

\subsection{Results}

The main achievement of this paper is the following partial answer to \cref{question:general}.

\begin{theorem}\label{thm:main}
	The limit of a polarized Gorenstein Fano variety under a toric degeneration is $\QQ$-polarized by its anticanonical divisor if and only if the polarization on the original variety is given by its anticanonical line bundle.
\end{theorem}

This result essentially says the following. If one were able to construct a toric degeneration of an anticanonically polarized Gorenstein Fano variety, one immediately gets that the toric limit is an anticanonically polarized $\QQ$-Gorenstein Fano variety.
This fact however is not new, it can be found for example in \cite[Theorem 3.8]{AB}. The new insight is that one can never reach anticanonically polarized limit varieties using any other polarization on the original variety! This greatly restricts the possible answers to \cref{question:general}.

However, in general our limit variety would only be $\QQ$-Gorenstein Fano. By a standard fact of toric geometry, it will be Gorenstein Fano if and only if the associated polytope is a lattice polytope. One can summarize this as follows.

\begin{theorem}\label{thm:reflexive}
	Assume that the polarized Gorenstein Fano variety $(X, \L)$ admits a toric degeneration to the polarized toric variety $(X_\P, D_\P)$ associated to a rational convex polytope $\P$. Then the polytope $\P$ is reflexive if and only if it is a lattice polytope and the line bundle $\L$ is the anticanonical line bundle over $X$.
\end{theorem}

So the question that remains to be answered is the following.

\begin{question}\label{question:simplified}
	Does every anticanonically polarized Gorenstein Fano variety admit a flat projective degeneration to a (necessarily anticanonically polarized) Gorenstein Fano toric variety (whose associated polytope would necessarily be a lattice polytope)?
\end{question}

This question remains quite difficult but at least we can answer it in the case of partial flag varieties using results from \cite{S3}.

\begin{theorem}\label{thm:flag}
	Let $G$ be a complex classical group or $\GG$. Then for any parabolic $P\subseteq G$ the partial flag variety $(G/P, \omega_{G/P}^{-1})$ admits a flat projective degeneration to an anticanonically polarized Gorenstein Fano toric variety.
\end{theorem}

The proof of \cref{thm:main} is not straightforward and requires methods from different branches of mathematics like toric geometry, polyhedral geometry and algebraic geometry. Its schematics are depicted in \cref{fig:proof}.

Apart from the results by Batyrev and Nill, we will use methods from Ehrhart theory\,---\,especially a beautiful result by Hibi \cite{Hi}\,---\,to translate the desired properties of the polarized toric variety into combinatorial properties of the Ehrhart polynomial of its associated polytope. As a by product we will prove the following generalization of a result that is well known for lattice polytopes.

\begin{theorem}\label{thm:toricehrhart}
	Let $\P\subseteq\RR^d$ be a full-dimensional rational convex polytope. Let $X_\P$ denote the associated normal projective toric variety and $D_\P$ the associated torus invariant $\QQ$-Weil divisor. Then
	\[\#(n\P\cap\ZZ^d) = \chi(X_\P, \O_{X_\P}(\lfloor nD_\P\rfloor))\]
	and	
	\[\#(\interior n\P \cap \ZZ^d) = \chi(X, \O_{X_\P}(\lceil nD_\P\rceil + K_{X_\P}))\]
	for all $n \in \NN$.
\end{theorem}

Using Invariance of Euler Characteristic in flat projective families (see \cite[Theorem 24.7.1]{V}) we will be able to deduce the following statement.

\begin{theorem}\label{cor:twopolynomials}
	Under the assumptions of \cref{thm:main}, the Hilbert polynomial associated to $X$ and $\L$ coincides with the Ehrhart quasi-polynomial associated to $\P$.
\end{theorem}

It has been known for quite some time that Serre Duality on toric varieties implies Ehrhart-Macdonald Reciprocity (see for example \cite[11.12.4]{D}). However, this theorem allows us to conclude that for a broad class of varieties the reverse implication also holds true. Hence it is indeed possible to prove one of the most famous results of algebraic geometry via lattice point counting!

The final step of our proof uses Serre Vanishing, Kodaira Vanishing for Rational Singularities and two results by Elkik \cite{E} on rational singularities in flat families. The main ingredient is the following observation.

\begin{theorem}\label{thm:vanishing}
	Let $X$ be a Gorenstein Fano variety of dimension $d$ that has rational singularities and let $\L$ be an ample line bundle. Then the line bundle $\L$ is isomorphic to the anticanonical line bundle $\omega_X^{-1}$ if and only if the Hilbert polynomial $P_\L(n) := \chi(X, \L^n)$ of $\L$ fulfills 
	\[ P_\L(n) = (-1)^d P_\L(-n-1)\]
	for all $n\in\NN$.
\end{theorem}

Our statement bears resemblance to a result by Kaveh and Villella in \cite{KV}, who were able to classify {\em anticanonical objects} in families of polyhedra associated to flag varieties purely via combinatorial conditions. However, their result needs stronger assumptions like Minkowski property of the occurring polytopes, which we do not need. 

During our proof we stumbled upon a delicate detail. First of all, the limit divisors in our setting share an interesting property.

\begin{theorem}\label{thm:dpweil}
	Under the assumptions of \cref{thm:main}, the divisor $D_\P$ is Weil.
\end{theorem}

Additionally, we noticed in \cref{thm:twopolynomials} that the Ehrhart quasi-polynomial of the polytope $\P$ must be a polynomial. Both of these properties can be thought of as integrality properties and we propose that they are in fact connected. This leads us to the following opinion.

\begin{conj}\label{conj:quasilattice}
	Let $\P$ be a full-dimensional rational convex polytope with associated toric variety $(X_\P, D_\P)$. Then the Ehrhart quasi-polynomial of $\P$ is a polynomial if and only if $D_\P$ is Weil.
\end{conj}

Unfortunately, we were not able to state a proof of this claim.

\subsection{Structure}

We will start our work by recalling the needed results from Ehrhart Theory and a proof of \cref{thm:toricehrhart} in \cref{sec:ehrhart}. The introduction and study of dual-integral polytopes will be done in \cref{sec:dual}. In \cref{sec:degen} we will formalize the concept of embedded toric degenerations and prove \cref{thm:dpweil,cor:twopolynomials} as well as reprove the known implication of \cref{thm:main}. \cref{sec:vanishing} is dedicated to the proof of \cref{thm:vanishing} which will allow us to prove \cref{thm:main,thm:reflexive} in \cref{sec:main}. We will finish this paper with a brief recap of the flag variety case from \cite{S2,S3}, proving \cref{thm:flag}. Especially, we will argue why the proof of \cref{thm:main} in the flag variety case\,---\,which has been presented in \cite{S2}\,---\,is easier than in the general case.

\subsection{Acknowledgments}

This paper is a rewritten excerpt from my PhD thesis \cite{S1} under the supervision of Peter Littelmann. I am very grateful to Peter Littelmann, Michel Brion and Xin Fang for many helpful discussions, for introducing me to the various mathematical concepts involved in this paper and for their invaluable, continued support. Additionally, I would like to thank my PhD referees Ghislain Fourier and Kiumars Kaveh for their kind and helpful comments.

\section{Preliminaries}

Let us first recall some basic facts on toric varieties. By a result of Sumihiro in \cite{Su}, every normal toric variety can be realized as the toric variety associated to a pointed rational polyhedral fan. Additionally, by \cite[Theorem 6.2.1]{CLS} a normal toric variety is projective if and only if the corresponding fan is the normal fan of a (full-dimensional) rational convex polytope. To any full-dimensional rational convex polytope $\P$ we associate a normal projective toric variety by $X_\P := X_{\Sigma_\P}$ via the normal fan $\Sigma_\P$ of $\P$.

Additionally, polytopes have a strong connection with torus invariant divisors. Let $X_\Sigma$ be a normal toric variety. Then there exist torus invariant prime divisors $D_\rho$ on $X_\Sigma$ for every ray $\rho\in\Sigma(1)$. These divisors generate the divisor class group of $X_\Sigma$ (see \cite[First Proposition of Section 3.4]{F}). So we can choose a torus invariant representative in every divisor class.

The canonical divisor has a particularly nice representative. By \cite[First Proposition of Section 4.3]{F} we have $K_{X_\Sigma} \sim -\sum_{\rho\in\Sigma(1)}D_\rho$.

Furthermore, divisors on toric varieties correspond to polytopes in the following way. 
We will always denote by $u_\rho$ the \textit{primitive ray generator} of the ray $\rho\in\Sigma(1)$, i.e. the unique lattice vector in $\rho$ whose coordinates are coprime.

Let $X_\Sigma$ be a normal toric variety and let $\P$ be a polytope such that $\Sigma$ is a refinement of the normal fan $\Sigma_\P$ of $\P$. We can always write $\P$ as
\[ \P = \{ x\in\RR^d \mid \langle x, u_\rho \rangle \leq b_\rho \text{ for all } \rho\in\Sigma(1)\},\]
where $b_\rho = \max_{x\in\P}\langle x,u_\rho\rangle\in\QQ$. Then we associate to $\P$ the torus invariant $\QQ$-Weil divisor $D_\P := \sum_{\rho\in\Sigma(1)} b_\rho D_\rho$.

Vice versa, to any torus invariant $\QQ$-Weil divisor $D = \sum_{\rho\in\Sigma(1)}a_\rho D_\rho$ we associate the (possibly empty) rational convex polytope 
\[ \P_D := \{ x\in\RR^d \mid \langle x, u_\rho \rangle \leq a_\rho \text{ for all } \rho\in\Sigma(1)\}.\]

The following observations are immediate.

\begin{prop}\label{prop:toricdivisors}
	Let $X_\Sigma$ be the normal toric variety corresponding to the pointed rational polyhedral fan $\Sigma\subseteq\RR^d$. Let $D$ be a torus invariant $\QQ$-Weil divisor on $X_\Sigma$ and let $\P$ be a rational convex polytope such that $\Sigma$ is a refinement of the normal fan $\Sigma_\P$.
	Then the following properties hold.
	\begin{enumerate}
		\item $\P_{D_\P} = \P$.
		\item $D_{\P_D} = D$ if $\Sigma_{\P_D} = \Sigma$.
		\item $\P_{kD} = k\P_D$ for any $k \in \RR_{\geq0}$.
		\item If $D \sim E$ for some torus invariant $\QQ$-Weil divisor $E$, we have $\P_E = \P_D + s$ for some $s \in\ZZ^d$.
	\end{enumerate}
\end{prop}

One divisor on projective normal varieties behave particularly well (see \cite[Propositions 4.2.10 and 6.1.10]{CLS}).

\begin{prop}\label{prop:dpample}
	Let $X_\P$ be the normal projective toric variety associated to the rational convex polytope $\P$. Then the divisor $D_\P$ is $\QQ$-Cartier and ample.
\end{prop}

To go from $\QQ$-Weil divisors to Weil divisors, we will need a rounding operation.

\begin{defi}
	Let $X_\Sigma$ be a normal toric variety and $D = \sum_{\rho\in\Sigma(1)}$ a $\QQ$-Weil divisor on $X_\Sigma$. The \textbf{round-down} of $D$ is defined as $\lfloor D \rfloor := \sum_{\rho\in\Sigma(1)} \lfloor a_\rho \rfloor D_\rho$. The \textbf{round-up} $\lceil D \rceil$ is defined analogously.
\end{defi}

For easier notation later-on we will define a similar operation for polytopes.

\begin{defi}
	Let $\P$ be a rational convex polytope with normal fan $\Sigma_\P$. Let $u_\rho$ denote the primitive ray generators of $\Sigma_\P$. Then there exist unique rational coefficients $b_\rho\in\QQ$ such that
	\[\P = \{ x\in\RR^d \mid \langle x, u_\rho \rangle \leq b_\rho \text{ for all } \rho\in\Sigma_\P(1)\}. \]
	The \textbf{round-down} of $\P$ is defined as
	\[ \P = \{ x\in\RR^d \mid \langle x, u_\rho \rangle \leq \lfloor b_\rho \rfloor \text{ for all } \rho\in\Sigma_\P(1)\}. \]
	The \textbf{round-up} of $\P$ is defined analogously.
\end{defi}

The following connection is clear.

\begin{prop}\label{prop:toricdivisors2}
	Let $X_\Sigma$ be a normal toric variety and $D$ a torus invariant $\QQ$-Weil divisor on $X_\Sigma$. Then $\lfloor \P_D\rfloor = \P_{\lfloor D \rfloor}$ and $\lceil \P_D\rceil = \P_{\lceil D \rceil}$.
\end{prop}

The connection between divisors and polytopes gives us possibilities to answer questions about cohomology. As an example we will state \cite[Proposition 4.3.3]{CLS}.

\begin{prop}\label{prop:toricehrhartweil}
	Let $X_\Sigma$ be the normal toric variety and $D$ a torus invariant Weil divisor on $X_\Sigma$. Then 
	\[ h^0(X_\Sigma, \O_{X_{\Sigma}}(D)) = \#(\P_D\cap\ZZ^d). \]
\end{prop}

This standard fact will motivate the next section.

\section{Ehrhart Theory}\label{sec:ehrhart}

Given a subset $S\subseteq\RR^d$ and a positive integer $n$, one might be interested in the number of lattice points in the dilation $nS$\,---\,i.e. the cardinality of $nS\cap\ZZ^d$. It turns out that there lies a beautiful theory behind this simple question if one starts with a convex rational polytope\,---\,called {\em Ehrhart Theory}. A well-written introduction into this theory is given in \cite{BR}. The main object of this theory is the following.

\begin{nota}\label{defi:ehrhart}
	Let $S \subseteq\RR^d$ be an arbitrary subset. Then for every integer $n\in\NN$ we will denote the number of lattice points in the $n$-th dilation of $S$ by $L_S(n)$, i.e.
	\[L_S(n) := \#(nS\cap\ZZ^d).\]
\end{nota} 

To formulate the birth result in this theory\,---\,called {\em Ehrhart-Macdonald Reciprocity}\,---\,, we need the following definition.

\begin{defi}
	A \textbf{quasi-polynomial} over $\RR$ is a function $f\colon\RR\to\RR$ that can be written as
	\[ f(x) = a_d(x) x^d + a_{d-1} (x)x^{d-1} + \ldots + a_1(x)x + a_0(x)  \]
	for some periodic functions $a_0, \ldots, a_d$ with integral period and $a_d \not\equiv 0$.  We call $d$ the \textbf{degree} of $f$.
\end{defi}

\begin{rem}
	An equivalent definition would be that a function ${f\colon\RR\to\RR}$ is called a quasi-polynomial if there exists an integer $T$ and polynomials \sloppy{$f_1, \ldots, f_T \in\RR[x]$} such that 
	\[  f(n) = f_{i}(n) \hspace{10pt}\text{ if } i\equiv n \mod T \]
	for every integer $n\in\NN$.
\end{rem}

We can now state a beautiful result, found for example in \cite[Theorem 4.1]{BR}.

\begin{theorem}[Ehrhart-Macdonald Reciprocity]\label{thm:emr}
	Let $\P$ be a rational convex polytope. Then there exists a quasi-polynomial $l_\P$ of degree $\dim\P$\,---\,called the \textbf{Ehrhart quasi-polynomial}\,---\,such that
	\[l_p(n) = L_\P(n) \text{ for all } n \in\NN.\]
	Any such quasi-polynomial $l_\P$ fulfills 
	\[ l_\P(-n) = (-1)^{\dim\P}L_{\interior \P}(n) \text{ for all } n\in\NN. \]
	Furthermore, if $\P$ is a lattice polytope, $l_\P$ can be chosen to be a rational polynomial.
\end{theorem}

\begin{rem}
	Since the period of a polynomial is not quite unique and quasi-polynomials are not uniquely determined by their values on integers, for a given polytope $\P$ we have many different quasi-polynomials $l_\P$ fulfilling Ehrhart-Macdo\-nald Reciprocity. Thankfully, we are generally only interested in evaluating those polynomials on integers\,---\,and these values are unique. So we will just call any quasi-polynomial {\em the} Ehrhart quasi-polynomial of $\P$ if it fulfills Ehrhart-Macdo\-nald Reciprocity and among all those quasi-polynomials there exists none of strictly smaller period. To simplify notation we will denote this quasi-polynomial by $L_\P$ too.
\end{rem}

\begin{ex}
	Let $\P = [0, 1/2] \subset \RR$. Then $L_\P(n) = \lfloor\frac{n}{2}\rfloor$, so one could choose for example
	\[ l_\P(x) = \frac{1}{2}x + 1-\frac{1}{2}\sin^2\left(\frac{\pi x}{2}\right) \text{ or }l_\P(x) = \frac{1}{2}x + \frac{3+\cos\pi x}{4}.\]
\end{ex}

We have seen that every lattice polytope will have an Ehrhart polynomial but there are also non-lattice convex polytopes whose Ehrhart quasi-polynomial is a polynomial. So the following category of polytopes should be quite interesting.

\begin{defi}
	A \textbf{quasi-lattice polytope} is a rational convex polytope whose Ehrhart quasi-polynomial is a polynomial.
\end{defi}

Examples of quasi-lattice polytopes can be found for example in \cite[Section 5]{S2}.

We will conclude this overview by proving the following new observation.

\begin{theorem*}[\cref{thm:toricehrhart}]
	Let $\P\subseteq\RR^d$ be a full-dimensional rational convex polytope. Let $X_\P$ denote the associated normal projective toric variety and $D_\P$ the associated torus invariant $\QQ$-Weil divisor. Then
	\[\#(n\P\cap\ZZ^d) = \chi(X_\P, \O_{X_\P}(\lfloor nD_\P\rfloor))\]
	and	
	\[\#(\interior n\P \cap \ZZ^d) = \chi(X, \O_{X_\P}(\lceil nD_\P\rceil + K_{X_\P}))\]
	for all $n \in \NN$.
\end{theorem*}

\begin{rem}
	\cref{thm:toricehrhart} would in theory also provide a method to compute the Ehrhart quasi-polynomial of an arbitrary rational convex polytope via cohomology groups\,---\,although in practice this might be quite challenging.
\end{rem}

Before we state the proof of this result, we show the following useful identities.

\begin{prop}\label{prop:round2}
	Let $\P\subseteq\RR^d$ be a full-dimensional rational convex polytope and let $X_\P$ denote the associated normal projective toric variety. Then 
	\begin{enumerate}
		\item $L_\P(n) = L_{\lfloor n\P\rfloor}(1)$,
		\item $L_{\interior \P}(n) = L_{\interior n\P}(1) = L_{\interior \lceil n\P\rceil}(1) = L_{\P_{\lceil nD_\P \rceil + K_{X_\P}}}(1)$
	\end{enumerate}
	for all $n\in\NN$.
\end{prop}

\begin{proof}
	Let $n\in\NN$. For every ray $\rho$ in the normal fan $\Sigma_\P$ let us denote its primitive ray generator by $u_\rho$. We know that $n\P$ can be written as
	\[ n\P = \{  x \in \RR^d \mid \langle x, u_\rho \rangle \leq nb_\rho \text{ for all } \rho \in\Sigma_\P(1)     \} \]
	for some rational numbers $b_\rho\in\QQ$, $\rho\in\Sigma_\P(1)$.
	
	For the first property, notice that for any lattice point $x\in\ZZ^d$ and any ray $\rho\in\Sigma_\P(1)$, the number $\langle u_\rho, x\rangle$ will be an integer, hence 
	\[   \langle x,u_\rho\rangle \leq nb_\rho \hspace{10pt}\Leftrightarrow\hspace{10pt} \langle x,u_\rho\rangle \leq \lfloor nb_\rho\rfloor,\]
	which proves the claim
	
	The first equality of the second part is clear since
	\begin{align*}
	x \in n\interior\P &\hspace{10pt}\Leftrightarrow\hspace{10pt} \left\langle\frac{1}{n}\cdot x, u_\rho \right\rangle < b_\rho \text{ for all } \rho \in\Sigma_\P(1) \\
	&\hspace{10pt}\Leftrightarrow\hspace{10pt} \langle x, u_\rho \rangle < nb_\rho\text{ for all } \rho \in\Sigma_\P(1)\\
	&\hspace{10pt}\Leftrightarrow\hspace{10pt} x \in \interior n\P
	\end{align*}
	for every $x\in\RR^d$.
	
	The second equality follows for the same reason as the first property (on the round-down), since for any $x \in\ZZ^d$ we have
	\[   \langle x,u_\rho\rangle < nb_\rho \hspace{10pt}\Leftrightarrow\hspace{10pt} \langle x,u_\rho\rangle < \lceil nb_\rho\rceil\]
	for any ray $\rho\in\Sigma_\P(1)$ because the scalar products are integers.
	
	For the third equality, recall that the canonical divisor is linearly equivalent to $-\sum_{\rho\in\Sigma_\P(1)}D_\rho$. So the two polytopes in the equation can be calculated as
	\[ \lceil n\P\rceil = \{  x\in\RR \mid \langle x, u_\rho \rangle \leq \lceil nb_\rho \rceil \text{ for all } \rho \in\Sigma_\P(1)\}   \]
	and 
	\[ \P_{\lceil nD_\P\rceil+K_{X_\P}} =  \{  x\in\RR \mid \langle x, u_\rho \rangle \leq \lceil nb_\rho \rceil - 1 \text{ for all } \rho \in\Sigma_\P(1)\}.  \]
	So for any $x\in\ZZ^d$ we get 
	\[   \langle x,u_\rho\rangle < \lceil nb_\rho\rceil \hspace{10pt}\Leftrightarrow\hspace{10pt} \langle x,u_\rho\rangle \leq \lceil nb_\rho\rceil-1\]
	because the scalar products are integers again.
\end{proof}

We are now able to prove the important formulae.

\begin{proof}[Proof of \cref{thm:toricehrhart}]
	Let us start with the first equality. The divisor $nD_\P$ will be $\QQ$-Cartier and ample, since $D_\P$ is $\QQ$-Cartier and ample (see \cref{prop:dpample}). So by Demazure Vanishing (see \cite[Theorem 9.3.5]{CLS}) we know that all higher cohomology groups of the round-down divisor $\lfloor nD_\P \rfloor$ must vanish. We can thus apply \cref{prop:toricehrhartweil} to the torus invariant Weil divisor $\lfloor nD_\P \rfloor$ to get
	\[  \chi(X_\P, \O_{X_\P}(\lfloor nD_\P \rfloor)) = h^0(X_\P, \O_{X_\P}(\lfloor nD_\P \rfloor)) = L_{\P_{\lfloor nD_\P \rfloor}} (1).  \]
	But because of \cref{prop:toricdivisors2} this polytope is equal to 
	\[ \P_{\lfloor nD_\P \rfloor} = \lfloor \P_{nD_\P}\rfloor = \lfloor n\P_{D_\P}\rfloor = \lfloor n\P\rfloor,   \]
	which proves the first claim using the observation
	\[ L_{\lfloor n\P\rfloor}(1) = L_\P(n)\]
	from \cref{prop:round2}. 
	
	Let us now consider the second case. By the toric version of the Kawamata-Viehweg Vanishing Theorem (see \cite[Corollary 2.5]{M} for a proof without Serre Duality) we know that all higher cohomology groups of the torus invariant Weil divisor $\lceil nD_\P\rceil + K_{X_\P}$ vanish. Using \cref{prop:toricehrhartweil} again we get
	\[  \chi(X_\P, \O_{X_\P}(\lceil nD_\P\rceil + K_{X_\P})) = h^0(X_\P, \O_{X_\P}(\lceil nD_\P\rceil + K_{X_\P})) = L_{\P_{\lceil nD_\P\rceil + K_{X_\P}}} (1).  \]
	By our previous observation in \cref{prop:round2} we know that this number can be rewritten as $L_{\interior \P}(n)$ which had to be proven.
\end{proof}

\section{Dual-Fano and Dual-Integral Polytopes}\label{sec:dual}

Before introducing some new classes of polytopes, let us introduce a shorthand notation.

\begin{nota}
	Let $v\in\RR^d$ and $b\in\RR$. Then we define the halfspaces
	\[ H^+_{v,b} := \{ x \in \RR^d\mid \langle x,v\rangle \geq b   \} \text{ and } H^-_{v,b} := \{  x\in\RR^d\mid\langle x, v\rangle \leq b\} \]
	as well as the affine hyperplane
	\[ H_{v,b} := H^+_{v,b}\cap H^-_{v,b}.  \]
	If $b=0$ we will just write $H^+_v := H^+_{v,0}$, $H^-_v := H^-_{v,0}$ and $H_v := H_{v,0}$.
\end{nota}

We will now recall the standard duality results on polytopes.

\begin{defi}
	Let $S\subseteq\RR^d$ be a set. The \textbf{(polar) dual} $S^*$ of the set $S$ is defined as 
	\[ \P^* := \{ x \in \RR^d \mid \langle x,s \rangle \leq 1 \text{ for all } s \in S  \}.   \]
\end{defi}

We will mostly omit the word {\em polar} when speaking about dual polytopes.

The following is a collection of standard facts that can be found for example in \cite[Theorem 2.11]{Z}).

\begin{theorem}\label{thm:polardual}
	Let $\P = \bigcap_{i=1}^r H^-_{\alpha_i,b_i}\subseteq\RR^d$ be a convex polytope with vertices $v_1, \ldots, v_s$.
	\begin{enumerate}
		\item $\P \subseteq (\P^*)^*$ and equality holds if and only if $0 \in \P$.
		\item If $0 \in \interior\P$ (implying that $\P$ is full-dimensional), then the $b_i$'s can be chosen non-zero. In this case, the polar dual $\P^*$ is a convex polytope and it can be calculated as \[\P^* = \conv(b_1^{-1}\alpha_1, \ldots, b_r^{-1}\alpha_r) = \bigcap_{i=1}^s H^-_{v_i,1}.\]
		\item If $0 \in \interior \P$ then $\P$ is rational if and only if $\P^*$ is rational.
		\item For every $\lambda \in \RR_{>0}$ we have $(\lambda\P)^* = \lambda^{-1}\P^*$.
		\item If $\P^*$ is a convex polytope, there exists an inclusion-reversing bijection between the faces of $\P$ and the faces of $\P^*$.
	\end{enumerate}
\end{theorem}

\begin{defi}
	A convex polytope $\P$ is called \textbf{reflexive} if both $\P$ itself and its polar dual $\P^*$ are lattice polytopes.
\end{defi}

Since this notion is too rigid for our applications we want to look at some weaker properties.

\begin{defi}\label{defi:dualfano}
	\begin{enumerate}
		\item A convex polytope is called a \textbf{quasi-reflexive} polytope if there exists a lattice point $p$ such that the translated polytope $\P-p$ is reflexive.
		\item A convex polytope is called a \textbf{Fano} polytope if its vertices are primitive lattice vectors.
		\item A convex polytope $\P$ is called a \textbf{dual-Fano} polytope there exists a lattice point $p$ such that the dual of the translated polytope $\P-p$ is a Fano polytope.
		\item A convex polytope $\P$ is called a \textbf{dual-integral} polytope if there exists a lattice point $p$ such that the dual of the translated polytope $\P-p$ is a lattice polytope.
	\end{enumerate}
\end{defi}

\begin{rem}
	While the notion of Fano polytopes is quite standard (see for example \cite{N}), the other notions are not. But the naming should be quite self-explanatory and it will be a useful shorthand.
	
	It should be noted that these definitions deviate from the definitions in \cite{S1}. In fact, we not only changed the name of dual-integral polytopes from {\em weakly dual-Fano} polytopes, but we also incorporated the lattice point translations into the definition and added the new definition of quasi-reflexive polytopes.
\end{rem}

\begin{rem}\label{rem:dualfano}
	Notice that every quasi-reflexive, dual-Fano or dual-integral polytope must contain a lattice point in its interior that we can chose as the translation vector. This is due to the fact that after translation by a lattice vector the translated dual of the translated polytope must be bounded, so it must contain the origin in its interior. Additionally, by \cref{thm:polardual} each of these polytopes must be rational since they are the lattice point translations of the dual of a lattice (hence rational) polytope.
\end{rem}

Before we get into the details, let us look at an example.

\begin{ex}\label{ex:dualfano}
	In \cref{fig:dualfano} we see the sketch of three different polytopes in $\RR^2$ and their dual polytopes. We will notice in this example that the three classes of polytopes introduced before are distinct.
	
	The first polytope is given by the inequalities
	\[\{(x,y)\in\RR^2\mid x\geq -1, y\geq -1, 2x+3y\leq 1\} \]
	and we clearly see from the sketch that it is reflexive, dual-Fano and dual-integral.
	
	The second polytope is given by the inequalities
	\[\{(x,y)\in\RR^2\mid x\geq -1, y\geq -1, x+3y\leq 1\} \]
	and we clearly see from the sketch that it is dual-Fano and dual-integral. However, since the upper left vertex is not integral, this polytope is not reflexive.
	
	Finally, the third polytope is given by the inequalities
	\[\{(x,y)\in\RR^2\mid x\geq -1, y\geq -1, 3x+3y\leq 1\}. \]
	From the sketch we see that it is dual-integral. But it is not a lattice polytope, so it cannot be reflexive. Additionally the upper right vertex of its dual polytope is the point $(3,3)$ which is an integral multiple of the lattice point $(1,1)$. Hence the dual polytope will not be Fano.
\end{ex}

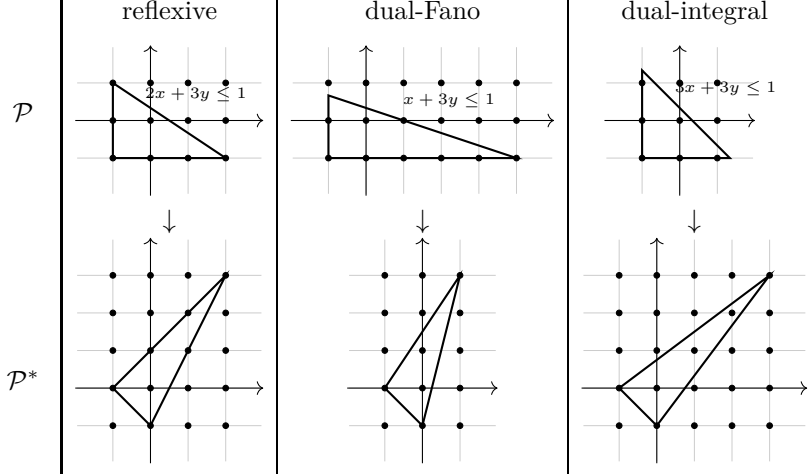
\begin{figure}[htb]\centering
	\caption{Sketch of the three different polytopes in \cref{ex:dualfano} and their dual polytopes. The first one is reflexive, dual-Fano and dual-integral; the second one is dual-Fano and dual-integral but not reflexive; the third one is only dual-integral.}\label{fig:dualfano}
	
	\begin{tabular}{c|c|c|c}
		&reflexive&dual-Fano&dual-integral\\
		\begin{tikzpicture}
		\node (1) at (0,1) { };
		\node (2) at (0,-1) { };
		\node (3) at (0,0) {$\P$};
		\end{tikzpicture}&
		\begin{tikzpicture}[scale=.5]
		\draw[step=1,help lines,black!20] (-1.95,-1.95) grid (2.95,1.95);
		\draw[->] (-2,0) -- (3,0);
		\draw[->] (0,-2) -- (0,2);
		\foreach \Point in {(-1,-1), (-1,0), (-1,1), (0,-1), (0,0), (0,1) , (1,-1), (1,0), (1,1), (2,-1), (2,0), (2,1)}
		\draw[fill=black] \Point circle (0.075);
		\coordinate (1) at (-1,1);
		\coordinate (2) at (-1,-1);
		\coordinate (3) at (2,-1);
		\draw[thick] (1) -- (2) -- (3) -- node[xshift=10pt,yshift=10pt] {\tiny $2x+3y\leq1$} cycle;
		\end{tikzpicture}
		&
		\begin{tikzpicture}[scale=.5]
		\draw[step=1,help lines,black!20] (-1.95,-1.95) grid (4.95,1.95);
		\draw[->] (-2,0) -- (5,0);
		\draw[->] (0,-2) -- (0,2);
		\foreach \Point in {(-1,-1), (-1,0), (-1,1), (0,-1), (0,0), (0,1) , (1,-1), (1,0), (1,1), (2,-1), (2,0), (2,1), (3,-1), (3,0), (3,1), (4,-1), (4,0), (4,1)}
		\draw[fill=black] \Point circle (0.075);
		\coordinate (1) at (-1,2/3);
		\coordinate (2) at (-1,-1);
		\coordinate (3) at (4,-1);
		\draw[thick] (1) -- (2) -- (3) -- node[xshift=10pt,yshift=10pt] {\tiny $x+3y\leq1$} cycle;
		\end{tikzpicture}
		&
		\begin{tikzpicture}[scale=.5]
		\draw[step=1,help lines,black!20] (-1.95,-1.95) grid (1.95,1.95);
		\draw[->] (-2,0) -- (2,0);
		\draw[->] (0,-2) -- (0,2);
		\foreach \Point in {(-1,-1), (-1,0), (-1,1), (0,-1), (0,0), (0,1) , (1,-1), (1,0), (1,1)}
		\draw[fill=black] \Point circle (0.075);
		\coordinate (1) at (-1,4/3);
		\coordinate (2) at (-1,-1);
		\coordinate (3) at (4/3,-1);
		\draw[thick] (1) -- (2) -- (3) -- node[xshift=15pt,yshift=10pt] {\tiny $3x+3y\leq1$} cycle;
		\end{tikzpicture}\\
		
		&$\downarrow$&$\downarrow$&$\downarrow$\\
		
		\begin{tikzpicture}
		\node (1) at (0,1) { };
		\node (2) at (0,-1) { };
		\node (3) at (0,0) {$\P^*$};
		\end{tikzpicture}
		&
		\begin{tikzpicture}[scale=.5]
		\draw[step=1,help lines,black!20] (-1.95,-1.95) grid (2.95,3.95);
		\draw[->] (-2,0) -- (3,0);
		\draw[->] (0,-2) -- (0,4);
		\foreach \Point in {(-1,-1), (-1,0), (-1,1), (-1,2), (-1,3), (0,-1), (0,0), (0,1), (0,2), (0,3), (1,-1), (1,0), (1,1), (1,2), (1,3), (2,-1), (2,0), (2,1), (2,2), (2,3)}
		\draw[fill=black] \Point circle (0.075);
		\coordinate (1) at (-1,0);
		\coordinate (2) at (0,-1);
		\coordinate (3) at (2,3);
		\draw[thick] (1) -- (2) -- (3) --  cycle;
		\end{tikzpicture}
		&
		\begin{tikzpicture}[scale=.5]
		\draw[step=1,help lines,black!20] (-1.95,-1.95) grid (1.95,3.95);
		\draw[->] (-2,0) -- (2,0);
		\draw[->] (0,-2) -- (0,4);
		\foreach \Point in {(-1,-1), (-1,0), (-1,1), (-1,2), (-1,3), (0,-1), (0,0), (0,1), (0,2), (0,3), (1,-1), (1,0), (1,1), (1,2), (1,3)}
		\draw[fill=black] \Point circle (0.075);
		\coordinate (1) at (-1,0);
		\coordinate (2) at (0,-1);
		\coordinate (3) at (1,3);
		\draw[thick] (1) -- (2) -- (3) --  cycle;
		\end{tikzpicture}
		&	
		\begin{tikzpicture}[scale=.5]
		\draw[step=1,help lines,black!20] (-1.95,-1.95) grid (3.95,3.95);
		\draw[->] (-2,0) -- (4,0);
		\draw[->] (0,-2) -- (0,4);
		\foreach \Point in {(-1,-1), (-1,0), (-1,1), (-1,2), (-1,3), (0,-1), (0,0), (0,1), (0,2), (0,3), (1,-1), (1,0), (1,1), (1,2), (1,3), (2,-1), (2,0), (2,1), (2,2), (2,3), (3,-1), (3,0), (3,1), (3,2), (3,3)}
		\draw[fill=black] \Point circle (0.075);
		\coordinate (1) at (-1,0);
		\coordinate (2) at (0,-1);
		\coordinate (3) at (3,3);
		\draw[thick] (1) -- (2) -- (3) --  cycle;
		\end{tikzpicture}	
	\end{tabular}
\end{figure}

Judging from the example we might guess that there is an inclusion relation between the three classes of polytopes and we could even guess how a standard form of these different polytopes would look like. We will formalize this in the next theorem and its corollary.

\begin{theorem}\label{thm:dualfano}
	\begin{enumerate}
		\item Every quasi-reflexive polytope is dual-Fano.
		\item Every dual-Fano polytope is dual-integral.
		\item A polytope is quasi-reflexive if and only if it is a dual-integral lattice polytope.
	\end{enumerate}
\end{theorem}

\begin{proof}
	Let $\P\subseteq\RR^d$ be a convex polytope. We can assume $\P$ to be rational and of full-dimension since all properties occurring in the statements imply rationality and full-dimensionality. Let $\Sigma_\P$ denote its normal fan with primitive ray generators $u_\rho$, $\rho\in\Sigma(1)$, i.e. $u_\rho$ is the unique lattice vector in $\rho$ whose coordinates are coprime.
	
	For the first claim, assume that $\P$ is a quasi-reflexive polytope. By \cref{rem:dualfano} we know that there exists a lattice point $p$ in the interior of $\P$. Let $\Q := \P-p$. Then $\Sigma_\Q = \Sigma_\P$. We can find rational numbers $b_\rho$ such that
	\[ \Q = \{ x \in \RR^d \mid \langle x, u_\rho \rangle \leq b_\rho \text{ for all } \rho\in\Sigma_\P(1)\}. \]
	Because $\P$ and hence $\Q$ is a lattice polytope, the hyperplanes $H_{u_\rho, b_\rho}$ must contain lattice points, which requires all $b_\rho$ to be integers. Since $0\in\interior\Q$, all $b_\rho$ must be strictly positive. By \cref{thm:polardual} we know that vertices of the dual polytope are given by
	\[ \vert\Q* = \left\{\frac{u_\rho}{b_\rho}\,\,\middle|\,\,\rho\in\Sigma_\P(1)\right\}.\]
	Since $\Q^*$ must be a lattice polytope, we know that all $b_\rho$ must be equal to $1$ (because the $u_\rho$ are primitive). So we get
	\[ \vert\Q^* = \left\{u_\rho\,\,\middle|\,\,\rho\in\Sigma_\P(1)\right\},\]
	which means that $\Q^*$ is a Fano polytope, i.e. $\P$ is a dual-Fano polytope.

	The second claim is obvious.
	
	
	Notice that claims (i) and (ii) already prove one direction of claim (iii). But the other direction follows immediately from the definitions.
\end{proof}

Interestingly, from the proof of \cref{thm:dualfano} we get the following descriptions.

\begin{cor}\label{cor:dualfano}
	Let $\P\subseteq\RR^d$ be a full-dimensional rational convex polytope and let $u_\rho$, $\rho\in\Sigma_\P(1)$, denote the primitive ray generators of the normal fan $\Sigma_\P$ of $\P$.
	
	\begin{enumerate}
		\item The polytope $\P$ is dual-integral if and only if there exist strictly positive integers $k_\rho$, $\rho\in\Sigma_\P(1)$, such that
		\[ \P = p + \{ x\in\RR^d \mid \langle x, u_\rho \rangle \leq k_\rho^{-1} \text{ for all } \rho\in\Sigma_\P(1)\}\]
		for some $p\in\interior\P\cap\ZZ^d$.
		\item The polytope $\P$ is dual-Fano if and only if 
		\[ \P = p + \{ x\in\RR^d \mid \langle x, u_\rho \rangle \leq 1 \text{ for all } \rho\in\Sigma_\P(1)\}\]
		for some $p\in\interior\P\cap\ZZ^d$.
		\item The polytope $\P$ is quasi-reflexive if and only if it is a lattice polytope and
		\[ \P = p + \{ x\in\RR^d \mid \langle x, u_\rho \rangle \leq 1 \text{ for all } \rho\in\Sigma_\P(1)\}\]
		for some $p\in\interior\P\cap\ZZ^d$.
	\end{enumerate}
\end{cor}

In light of toric geometry we repeat the following observation. It can be found in in Nill's doctoral thesis \cite[Proposition 1.4]{N} and follows from a result by Batyrev \cite[Proposition 2.2.23]{B}. Recall that the anticanonical divisor on a normal toric variety $X_\Sigma$ is given by $-K_{X_\Sigma} = \sum_{\rho\in\Sigma(1)}D_\rho$, hence its associated polytope $\P_{-K_{X_\Sigma}}$ can be written as
\[ \P_{-K_{X_\P}} = \{ x\in\RR^d \mid \langle x, u_\rho \rangle \leq 1 \text{ for all } \rho\in\Sigma_\P(1)\}.\]

\begin{theorem}[Batyrev, Nill]\label{thm:toricfano}
	The map $(X_\P, -K_{X_\P}) \to \P_{-K_{X_\P}}$ induces a one-to-one correspondence between isomorphism classes of anticanonically polarized $\QQ$-Gorenstein Fano toric varieties and dual-Fano polytopes up to lattice translations. The same statement holds true for Gorenstein Fano toric varieties and reflexive polytopes.
\end{theorem}

We will conclude this overview with a beautiful result by Hibi \cite{Hi}.

\begin{theorem}[Hibi]\label{thm:hibi}
	A full-dimensional rational convex polytope $\P\subseteq\RR^d$ is dual-integral if and only if 
	\[  \#(n\P\cap\ZZ^d) = \#(\interior(n+1)\P\cap\ZZ^d) \]
	for all $n \in \NN$.
\end{theorem}

\begin{rem}\label{rem:hibi}
	This formulation of Hibi's result is not his original formulation but this version shows clearer which beautiful magic is actual happening in the background. First of all, Hibi did not use our notion of dual-integral polytopes. Secondly, for computational purposes it is helpful to notice that by Ehrhart-Macdonald Reciprocity (see \cref{thm:emr}) the condition of Hibi's Theorem is equivalent to the property that 
	\[ L_\P(n) = (-1)^dL_\P(-n-1)  \]
	for all $n \in \NN$\,---\,a condition that can be verified by studying the Ehrhart quasi-polynomial alone. This was his original statement.
\end{rem}

The following consequence is immediate.

\begin{cor}\label{cor:hibi}
	Every dual-integral polytope contains precisely one lattice point in its interior.
\end{cor}

From this observation we see that the lattice vectors in \cref{defi:dualfano} are unique. Additionally, the following notation is verified.

\begin{nota}
	Let $\P$ be a dual-integral polytope. The unique interior lattice point in $\P$ will be denoted by $p_\P$.
\end{nota}

We will conclude this section with an interesting little observation that does not generalize to higher dimensions.

\begin{theorem}
	Every two-dimensional quasi-lattice polytope is dual-integral if and only if it contains exactly one interior lattice point.
\end{theorem}

\begin{proof}
	The first direction is clear. So let $\P\subseteq\RR^2$ be a full-dimensional quasi-lattice polytope with $L_\P(-1) = \#(\interior \P\cap\ZZ^2) = 1$. By \cref{thm:emr} we can find rational coefficients $a$, $b$ and $c$ such that the Ehrhart polynomial of $\P$ is given by 
	\[ L_\P(x) = ax^2+bx+c. \]
	Since $L_\P(0) = 1$ we have $c=1$. Additionally, $L_\P(-1) = 1$ implies that $a = b$. So we see that 
	\[ L_\P(x) = ax(x+1) + 1 \]
	for some strictly positive rational coefficient $a$. Hence
	\[ L_\P(-n-1) = a(-n-1)(-n-1+1) + 1 = an(n+1) + 1 = L_\P(n),  \]
	so Hibi's \cref{thm:hibi} implies the claim.
\end{proof}

\section{Embedded Toric Degenerations}\label{sec:degen}

The following definition could be seen as the standard definition of toric degenerations.

\begin{defi}[First try]
	Let $X$ be a normal complex variety. We say that the variety $X$ admits a \textbf{toric degeneration} to the normal toric variety $X_\Sigma$ if there exists a complex variety $\X$ and a flat morphism $\X \to \AA^1$ such that all fibers $\X_t$, $t\in\AA^1$, are isomorphic to $X$ and the special fiber $\X_0$ is isomorphic to $X_\Sigma$.
\end{defi}

However, this definition would be to weak for our applications. So we need to strengthen it as follows. Additionally, we are only interested in the case of projective varieties.

\begin{defi}\label{defi:toricdegen}
	Let $X$ be a normal projective complex variety of dimension $d$ and let $D$ be an ample Cartier Weil divisor on $X$. Let $X_\P$ be the normal projective toric variety associated to the rational convex polytope $\P\subseteq\RR^d$. Let $D_\P$ denote the ample $\QQ$-Cartier $\QQ$-Weil divisor on $X_\P$ associated to $\P$ (see \cref{prop:dpample} for these properties of $D_\P$). We say that the pair $(X,D)$ admits a \textbf{toric degeneration} to the pair $(X_\P, D_\P)$ if there exists a complex variety $\X$ and a morphism 
	\[\pi \colon \X \to \AA^1\]
	such that
	\begin{enumerate}
		\item $\pi$ is projective and flat,
		\item $\X_t \simeq X$ for all $t \neq 0$ and $\X_0 \simeq X_\P$, and
		\item for every $n \in \NN$ there exists a divisorial sheaf $\F^{(n)}$ on $\X$ such that $\F^{(n)}|_{\X_t} \simeq \O_X(nD)$ and $\F^{(n)}|_{\X_0} \simeq \O_{X_\P}(\lfloor n D_\P \rfloor )$.
	\end{enumerate}
	The variety $X_\P$ is called the \textbf{limit} of $X$ under the degeneration $\pi$.
\end{defi}

\begin{rem}
	Whenever we say that {\em a variety admits a toric degeneration}, we mean this in the sense of \cref{defi:toricdegen} and not in the standard sense! We acknowledge that our terminology is not standard. Yet, we will see that basically all known examples of toric degenerations in the usual sense are also toric degenerations in our sense. The main reason behind this change is basically that we want to keep track of the embedding of our varieties. So one could think of our degenerations as {\em embedded toric degenerations}.
\end{rem}

\begin{rem}
	We will use the analogue definition of toric degenerations for pairs $(X, \L)$ where $\L$ is an ample line bundle on the normal projective variety $X$.
\end{rem}

\begin{rem}\label{rem:birational}
	This definition should not come completely unexpected. It has been proven by Anderson that toric degenerations in our sense arise naturally in the setting of Newton-Okounkov bodies (see \cite[Theorem 1, Corollary 2 and Lemma 3]{A}).
	
	In the case of flag varieties there is an even broader formalism. Fang, Fourier and Littelmann constructed toric degenerations via so called {\em birational sequences} in \cite{FaFL}. They proved in \cite[Theorem 6]{FaFL} that their toric degenerations satisfy the properties of \cref{defi:toricdegen}. Additionally, they were able to show that every birational sequence defines a valuation. Hence every toric degeneration via birational sequences can be seen as a toric degeneration associated to a Newton-Okounkov body. It is noticeable that basically all polytopes in representation theory can be constructed via birational sequences and hence are Newton-Okounkov bodies. We will make use of this fact in \cref{sec:flag}.
\end{rem}

\begin{defi}
	Let $\X\to\AA^1$ be a toric degeneration. The variety $\X$ is called the \textbf{degeneration space}. The fiber $\X_0$ is called the \textbf{special fiber} and the fiber $\X_1$\,---\,being isomorphic to any $\X_t$, $t\neq0,$\,---\,is called the \textbf{general fiber}.
\end{defi}

We will collect some useful facts of the degeneration space $\X$. It should be noted that the following three propositions also hold under the usual definition of toric degenerations.

\begin{prop}\label{prop:degenerationspace}
	Let $\X \to \AA^1$ be a toric degeneration in the sense of \cref{defi:toricdegen}. Then $\X$ is quasi-projective, normal and has rational singularities.
\end{prop}

\begin{proof}
	$\X$ is clearly quasi-projective since it is the domain of a projective morphism. Normality follows directly from Serre's criterion of normality since every fiber $\X_t$, $t\in\RR$, is normal (see \cite[Théorème 5.8.6]{EGAIV2}). Rationality of singularities follows from the fact that normal toric varieties have rational singularities (see \cite[Theorem 11.4.2]{CLS}). So by a result of Elkik (see \cite[Théorème 4]{E}) there exists an open neighborhood $0\in U\subseteq \AA^1$ such that all fibers $\X_t$, $t\in U$, have rational singularities. But all fibers are isomorphic (except for the special fiber), hence all fibers have rational singularities. This implies that $\X$ has rational singularities by another result of Elkik (see \cite[Théorème 5]{E}).
\end{proof}

\begin{prop}\label{prop:x0}
	Let $\X \to \AA^1$ be a toric degeneration in the sense of \cref{defi:toricdegen}. Then each fiber $\X_t$ is a principal prime divisor on $\X$.
\end{prop}

\begin{proof}
	By \cite[Chapter III, Corollary 9.6]{Ha} we know that $\X_t$ is a subvariety of codimension 1. The irreducibility is clear since $\X_t \simeq X$ or $\X_t \simeq X_\Sigma$. Notice that $\pi\in\O_\X(\X) \subseteq \CC(\X)$, so $\pi$ is a rational function on $\X$ and $\X_t = \div(\pi-t)$ is principal.
\end{proof}

\begin{prop}\label{prop:clx}
	Let $\X \to \AA^1$ be a toric degeneration in the sense of \cref{defi:toricdegen} and let $X$ be isomorphic to the general fiber. Then for any $t\neq0$ the restriction map $\D\mapsto \D|_{\X_t}$ induces an isomorphism of divisor class groups
	\[\Cl(\X) \simeq \Cl(X).\]
\end{prop}

\begin{proof}
	Consider the variety $X\times\AA^1$ and the natural projection $\tau\colon X\times\AA^1\to \AA^1$. Let $Z := \tau^{-1}(0)$.	
	Then the open set $U := \X\setminus\X_0$ is isomorphic to the open set $V := (X\times\AA^1 )\setminus Z$.
	
	By \cref{prop:x0} we know that the special fiber $\X_0$ is a principal prime divisor on $\X$. By the same arguments $Z$ is a principal prime divisor on $X\times\AA^1$. So repeated application of \cite[Propositions 6.5 and 6.6]{Ha} gives a sequence of isomorphisms
	\begin{align*} \Cl(\X) \to \Cl(\X)/\ZZ\lbrack\X_0\rbrack \to &\Cl(U)\\ &\downarrow\\ &\Cl(V) \to \Cl (X\times\AA^1)/\ZZ\lbrack Z\rbrack \to \Cl(X\times\AA^1) \to \Cl(X), \end{align*}
	where each step is either a restriction, identity or induced by the isomorphism between $X$ and the general fiber.	
\end{proof}

Although we allowed our limit divisor $D_\P$ to be a $\QQ$-Weil divisor, we will now prove that this divisor will in fact be integral if the Weil divisor on our original variety was Cartier. This in return restricts the polytopes that could appear in our setting. We will try to explain these restrictions at the end of this section.

\begin{theorem*}[\cref{thm:dpweil}]
	Let $X$ be a normal projective complex variety of dimension $d$ and let $D$ be an ample Cartier divisor on $X$. Let $X_\P$ be the normal projective toric variety associated to the rational convex polytope $\P\subseteq\RR^d$. Suppose the pair $(X, D)$ admits a toric degeneration to the toric pair $(X_\P, D_\P)$. Then $D_\P$ is a Weil divisor.
\end{theorem*}

\begin{proof}
	Notice that $\F^{(n)}$ is divisorial for every $n\in\NN$, so there exist divisors $\D_n$ on the degeneration space $\X$, $n\in\NN$, such that $\F^{(n)}\simeq\O_\X(\D_n)$. Fix $t\neq0$ and $n\in\NN$. Since $D$ is Cartier, we know that 
	\[\F^{(n)}|_{\X_t} \simeq \O_X(nD) \simeq \O_X(D)^{\otimes n} \simeq (\F^{(1)})^n|_{\X_t}.\]
	Translated to divisors this means 
	\[ \D_n|_{\X_t} \sim n\D_1|_{\X_t}.\]
	By \cref{prop:clx} this is only possible if $\D_n \sim n\D_1$.
	
	So there exists a Weil divisor $\D := \D_1$ on $\X$ such that $\F^{(n)} \simeq \O_\X(n\D)$ for every $n\in\NN$. Let $E$ denote the divisor on $X_\P$ that is isomorphic to $\D|_{\X_0}$ via the isomorphism $\X_0\simeq X_\P$.
	
	Since $\F^{(n)}|_{\X_0}\simeq\O_{X_\P}(\lfloor nD_\P \rfloor)$ for every $n\in\NN$, we have 
	\[ \lfloor nD_\P \rfloor \sim nE  \]
	for every $n\in\NN$. Chose $l\in\NN$ such that $lD_\P$ is Weil. Then
	\[ lD_\P \sim lE\]
	i.e. there exist rational functions $f,g\in\CC(X_\P)^\times$ such that 
	\[ \lfloor D_\P \rfloor = E+\div(f) \hspace{10pt}\text{ and }\hspace{10pt} lD_\P = lE + \div(g).    \]
	
	We want to show that $g = \lambda f^l$ for some $\lambda\in\CC$. Notice that 
	\[ D_\P = E + l^{-1}\div(g), \]
	hence \[E+\div(f) = \lfloor D_\P\rfloor = \lfloor E+l^{-1}\div(g) \rfloor = E+\lfloor l^{-1}\div(g)\rfloor.\]	
	So in conclusion we know that $\div(f) = \lfloor l^{-1}\div(g)\rfloor$. Since every coefficient of $\lfloor l^{-1}\div(g)\rfloor $ is smaller than or equal to the corresponding coefficient of $l^{-1}\div(g)$, we see that $l^{-1}\div(g)-\div(f)$ is an effective divisor. Equivalently, the difference $\div(g) - l\div(f) = \div(f^{-l}g)$ is an effective divisor.
	
	But this implies that $f^{-l}g$ is a rational function on $X_\P$ whose valuation $v_{D_i}(f^{-l}g)$ on every prime divisor $D_i$ is non-negative. This is only possible if $f^{-l}g$ is regular. But since $X_\P$ is projective, every regular function is constant, so there exist $\lambda \in\CC^\times$ such that $f^{-l}g=\lambda$\,---\,or equivalently $g = \lambda f^l$.
	
	In conclusion we see that $\div(g)=\div(\lambda f^l)=\div(f^l)$ and thus
	\[D_\P = E+l^{-1}\div(g) = E+l^{-1}\div(f^l) = E+\div(f)\]
	is a Weil divisor on $X_\P$.
\end{proof}

Let us now prove an important consequence of \cref{thm:toricehrhart}.

\begin{theorem}\label{thm:twopolynomials}
	Let $X$ be a normal projective complex Gorenstein variety of dimension $d$ and let $\L$ be an ample line bundle over $X$. Let $\P \subseteq\RR^d$ be a full-dimensional rational convex polytope. Suppose that the pair $(X, \L)$ admits a toric degeneration to the toric variety $(X_\P, D_\P)$. Then
	\[ \chi(X, \L^n) = \#(n\P\cap \ZZ^d) \hspace{10pt}\text{ and }\hspace{10pt} \chi(X, \L^n\otimes\omega_X) = \#(\interior n\P \cap \ZZ^d)\]
	for all $n \in \NN$.
\end{theorem}

\begin{proof}
	By \cref{thm:dpweil} we know that $D_\P$ is a Weil divisor on $X_\P$.
	
	By our assumption there exist divisorial sheaves $\F^{(n)}$ for every $n\in\NN$ such that 
	\[ \F^{(n)}|_{\X_t} \simeq \L^n \text{ for all } t\neq0 \text{ and } \F^{(n)}|_{\X_0} \simeq \O_{X_\P}(nD_\P).\]
	By \cref{thm:toricehrhart} we have
	\[ \#(n\P\cap\ZZ^d) = \chi(X_\P, \O_{X_\P}(nD_\P)).\]
	Since Euler characteristic is constant in flat projective families (see for example \cite[Theorem 24.7.1]{V}), we conclude
	\[ \#(n\P\cap\ZZ^d) = \chi(X_\P, \O_{X_\P}( nD_\P)) = \chi(X, \L^n)\]
	for every $n\in\NN$.	
	
	For the second set of equations let $\D$ denote the Weil divisor from the proof of \cref{thm:dpweil}, i.e. $\F^{(n)} = \O_\X(n\D)$ for all $n\in\NN$. Consider the divisorial sheaves $\O_\X(n\D+K_\X)$ for each $n\in\NN$. Recall that each fiber $\X_t$ is a principle prime divisor on $\X$ by \cref{prop:x0}. So the Adjunction Formula (see for example \cite[30.4.8]{V}) yields
	\[ \O_\X(n\D+K_\X)|_{\X_t}\simeq \O_{\X_t}((n\D+K_\X+\X_t)|_{\X_t}) \simeq \begin{cases}
	\O_{X_\P}(nD_\P+K_{X_\P})\text{ if } t=0,\\
	\L^n\otimes\omega_X \text{ else}
	\end{cases} \]
	since $K_X$ is Cartier. The claim follows from the same reasons as before.
\end{proof}

There exists a different formulation of this theorem.

\begin{theorem*}[\cref{cor:twopolynomials}]
	Let $X$ be a normal projective complex Gorenstein variety of dimension $d$ and let $\L$ be an ample line bundle over $X$. Let $\P \subseteq\RR^d$ be a full-dimensional rational convex polytope. Suppose that the pair $(X, \L)$ admits a toric degeneration to the toric variety $(X_\P, D_\P)$. Then the Ehrhart quasi-polynomial $L_\P$ of $\P$ coincides with the Hilbert polynomial $P_\L(n) := \chi(X, \L^n)$ of $X$ and $\L$, i.e.
	\[ L_\P(n) = P_\L(n) \]
	for all $n \in \ZZ$.
\end{theorem*}

\begin{proof}
	The claim is true for positive $n$ by \cref{thm:twopolynomials}. Since $L_\P$ is a quasi-polynomial that coincides with the polynomial $P_\L$ on all positive integers, it must be a polynomial itself. Since both polynomials $L_\P$ and $P_\L$ coincide on all positive integers, they must coincide on all integers.
\end{proof}

We want to save one important observation in this proof for later purposes. It does not even need the Gorenstein hypothesis.

\begin{cor}\label{cor:degenerationpolynomial}
	Let $\P$ be a rational convex polytope. Then the associated pair $(X_\P, D_\P)$ can only be the limit of a normal projective variety and an ample line bundle under a toric degeneration if the Ehrhart quasi-polynomial of $\P$ is a polynomial.
\end{cor}

We will return to this observation at the end of this section. Another consequence is the following.

\begin{rem}\label{rem:emrserre}
	It is already known (see for example \cite[11.12.4]{D}) that Serre Duality for toric varieties implies Ehrhart-Macdonald Reciprocity for lattice polytopes (notice that Cartier divisors on toric varieties correspond to lattice polytopes by \cite[Theorem 4.2.8]{CLS}). Our formulation in \cref{cor:twopolynomials} generalizes this in two regards.
	
	On one side, we could prove Ehrhart-Macdonald Reciprocity for arbitrary rational convex polytopes by using toric Serre Duality (see for example \cite[Exercise 9.3.5]{CLS} for a rational version) since \cref{thm:toricehrhart} implies
	\begin{align*}
	L_\P(-n) &= \chi(X_\P, \O_{X_\P}(\lfloor -nD_\P \rfloor)) = \chi(X_\P, \O_{X_\P}(-\lceil nD_\P \rceil)) \\&= (-1)^d\chi(X_\P, \O_{X_\P}(\lceil nD_\P \rceil+K_{X_\P})) = (-1)^dL_{\interior\P}(n).
	\end{align*}
	Here we additionally used that the Ehrhart quasi-polynomial of $\P$ and the Hilbert quasi-polynomial of $D_\P$ must coincide on all integers since they coincide on all positive integers.
	
	But most importantly, the reverse argument is also possible! Whenever we have a normal projective Gorenstein variety $(X, \L)$ that admits a toric degeneration to any (normal projective) toric pair $(X_\P, D_\P)$, we get
	\[ \chi(X, \L^{-n}) = L_\P(-n) = (-1)^d L_{\interior\P}(n) = (-1)^d \chi(X, \L^n\otimes\omega_X).\]
	This remarkable fact shows that Serre Duality can be proved by counting lattice points for a broad class of varieties. This fact is depicted in \cref{fig:emrserre}.
\end{rem}

\begin{figure}[htb]\centering
	\caption{Let $X$ be a normal projective complex Gorenstein variety, $\L$ an ample line bundle over $X$ and $\P$ a rational convex polytope such that $(X,\L)$ admits a toric degeneration to $(X_\P, D_\P)$. Then the sketched equalities hold for every $n\in\NN$.}\label{fig:emrserre}
	\begin{tikzpicture}
	\matrix (m) [matrix of math nodes,row sep=3em,column sep=7em,minimum width=2em]
	{
		\chi(X, \L^{-n})     & (-1)^d\chi(X, \L^n\otimes\omega_X) \\
		L_\P(-n) & (-1)^dL_{\interior \P}(n) \\};
	\draw[double distance = 3pt] (m-1-1) -- node[above] {\scriptsize Serre} node[below] {\scriptsize Duality} (m-1-2);
	\draw[double distance = 3pt] (m-1-2) -- (m-2-2);
	\draw[double distance = 3pt] (m-2-1) -- node[above] {\scriptsize Ehrhart-Macdonald} node[below] {\scriptsize Reciprocity} (m-2-2);
	\draw[double distance = 3pt] (m-1-1) -- (m-2-1);
	\end{tikzpicture}
\end{figure}

\section{A Unique Combinatorial Feature of the Anticanonical Bundle}\label{sec:vanishing}

This section is dedicated to proving the following central observation.

\begin{theorem*}[\cref{thm:vanishing}]
	Let $X$ be a Gorenstein Fano variety of dimension $d$ that has rational singularities and $\L$ an ample line bundle. Then the line bundle $\L$ is isomorphic to the anticanonical line bundle $\omega_X^{-1}$ if and only if the Hilbert polynomial of $\L$ fulfills 
	\[ P_\L(n) = (-1)^d P_\L(-n-1)\]
	for all $n\in\NN$.
\end{theorem*}

\begin{proof}
	Notice first that Serre Duality implies that
	\[ P_\L(-n-1) = \chi(X, \L^{-n-1}) = (-1)^d \chi(X, \L^{n+1}\otimes\omega_X).\]
	
	So the first implication is obvious, since 
	\[ P_{\omega_X^{-1}}(-n-1) = (-1)^d \chi(X, \omega_X^{-n-1}\otimes \omega_X) =(-1)^d\chi(X, \omega_X^{-n}) = (-1)^d P_{\omega_X^{-1}}(n)\]
	for every $n\in\NN$.
	
	For the other implication let us notice that Kodaira Vanishing for Rational Singularities (see \cite[Theorem 2.70]{KM}, which can be proved as a consequence of the Grauert-Riemenschneider Vanishing Theorem in \cite[Satz 2.3]{GR}) implies that 
	\[\chi(X, \L^{n+1}\otimes \omega_X) = h^0(X, \L^{n+1}\otimes \omega_X)\] 
	for every $n\in\NN$. So for the special case $n = 0$ the assumption on the Hilbert polynomial implies that 
	\begin{align*}h^0(X, \L\otimes\omega_X) &= \chi(X, \L\otimes\omega_X) = (-1)^d P_\L(-1) \\ &=P_\L(0) = \chi(X, \O_X) = h^0(X, \O_X) = 1.\end{align*}
	Chose a non-zero section $s \in H^0(X, \L\otimes\omega_X)$ and consider the natural morphism $f\colon \O_X \to \L\otimes\omega_X$ given by 
	\[f(U) \colon \O_X(U) \to (\L\otimes\omega_X)(U), \hspace{10pt} \xi \mapsto \xi \cdot s|_U,\]
	on every open set $U \subseteq X$.
	
	Since $\L\otimes\omega_X$ is a line-bundle, it is torsion-free. So the maps $f(U)$ and hence the morphism $f$ are injective. Thus we get a short exact sequence
	\[ 0 \to \O_X \stackrel{f}{\to} \L\otimes\omega_X \to \C \to 0  \]
	for some coherent sheaf $\mathcal{C}$. We want to show that $\C\simeq 0$. 
	Recall that Euler characteristic is additive on short exact sequences (see for example \cite[Lemma 2.5.2]{EGAIII1}). Hence tensoring with $\L^n$ and taking Euler characteristic yields
	\[\chi(X, \L^n\otimes\C) = \chi(X, \L^{n+1}\otimes\omega_X) - \chi(X, \L^n) = (-1)^dP_\L(-n-1) - P_\L(n) = 0 \]
	for every $n \in \NN$. 
	
	By the Serre Vanishing Theorem (see for example \cite[Proposition 2.2.2]{EGAIII1}) we can chose $n$ big enough such that the sheaf $\mathcal{C}\otimes\L^n$ is globally generated and its higher cohomologies vanish. Hence
	\[ h^0(X, \mathcal{C}\otimes\L^n) = \chi(X, \mathcal{C}\otimes\L^n) = 0. \]
	Since $\mathcal{C}\otimes\L^n$ is globally generated, this is only possible if $\mathcal{C}\otimes\L^n \simeq 0$, i.e. $\mathcal{C} \simeq 0$.
	
	In conclusion we see that $\L\otimes\omega_X \simeq \O_X$ holds\,---\,or equivalently $\L \simeq \omega_X^{-1}$, which proves the claim.
\end{proof}

\section{Anticanonically Polarized Toric Degenerations}\label{sec:main}

We are now able to prove our main result. A sketch of the proof is presented in \cref{fig:proof}.

\begin{proof}[Proof of \cref{thm:main}]		
	Let $(X, \L)$ be a polarized Gorenstein Fano variety and suppose that it admits a toric degeneration to the pair $(X_\P, D_\P)$, where $\P\subseteq\RR^d$ denotes a full-dimensional rational convex polytope, $X_\P$ the associated toric variety and $D_\P$ the associated torus invariant ample $\QQ$-Cartier Weil divisor on $X_\P$ (see \cref{prop:dpample,thm:dpweil}).
	
	Let us first assume that $\L\simeq\omega_X^{-1}$. Notice that the isomorphism from \cref{prop:clx} yields $\F^{(1)} \simeq \O_\X(-K_\X)$ and recall that $\X_0$ is a principal prime divisor on $\X$ by \cref{prop:x0}. So the Adjunction Formula (see for example \cite[Proposition 30.4.8]{V}) shows that 
	\[ K_{\X_0} = (K_\X+\X_0)|_{\X_0} \sim K_\X|_{\X_0}\]
	which implies 
	\[ \O_{X_\P}(D_\P) \simeq \F^{(1)}|_{\X_0} \simeq \O_{\X_0} (-K_{\X_0}) \]
	and thus $D_\P\sim-K_{X_\P}$.	
	Additionally, since $-K_{X_\P} \sim D_\P$ is $\QQ$-Cartier and ample, the toric limit variety $X_\P$ is $\QQ$-Gorenstein Fano.
	
	For the reverse implication let us assume that $D_\P\sim -K_{X_\P}$ and hence $X_\P$ is $\QQ$-Gorenstein Fano. By \cref{prop:toricdivisors} we know that
	\[ \P = \P_{D_\P} = \P_{-K_{X_\P}} + s\]
	for some lattice vector $s\in\ZZ^d$.	
	Since the toric variety $X_\P$ is $\QQ$-Gorenstein Fano, \cref{thm:toricfano} (following Nill \cite[Proposition 1.4]{N} and Batyrev \cite[Proposition 2.2.23]{B}) shows that the dual of the polytope $\P_{-K_{X_\P}}$ is Fano. Hence $\P$ is dual-integral (it would even be dual-Fano), so Hibi's \cref{thm:hibi} implies that the Ehrhart polynomial $L_\P$ of $\P$ fulfills
	\[ L_\P(n) = (-1)^dL_{\P}(-n-1)\]
	for all $n\in\NN$. 	
	We have seen in \cref{cor:twopolynomials} that this Ehrhart quasi-polynomial and the Hilbert polynomial $P_\L$ of $X$ and $\L$ coincide, hence
	\[ P_\L(n) = (-1)^dP_\L(-n-1) \]
	for every $n\in\NN$. Since $X_\P$ has rational singularities (see for example \cite[Theorem 11.4.2]{CLS}), Elkik's result \cite[Théorème 4]{E} proves that $X$ has rational singularities as well. So we can apply \cref{thm:vanishing}, which implies that $\L$ is isomorphic to $\omega_X^{-1}$. This concludes the proof.
\end{proof}

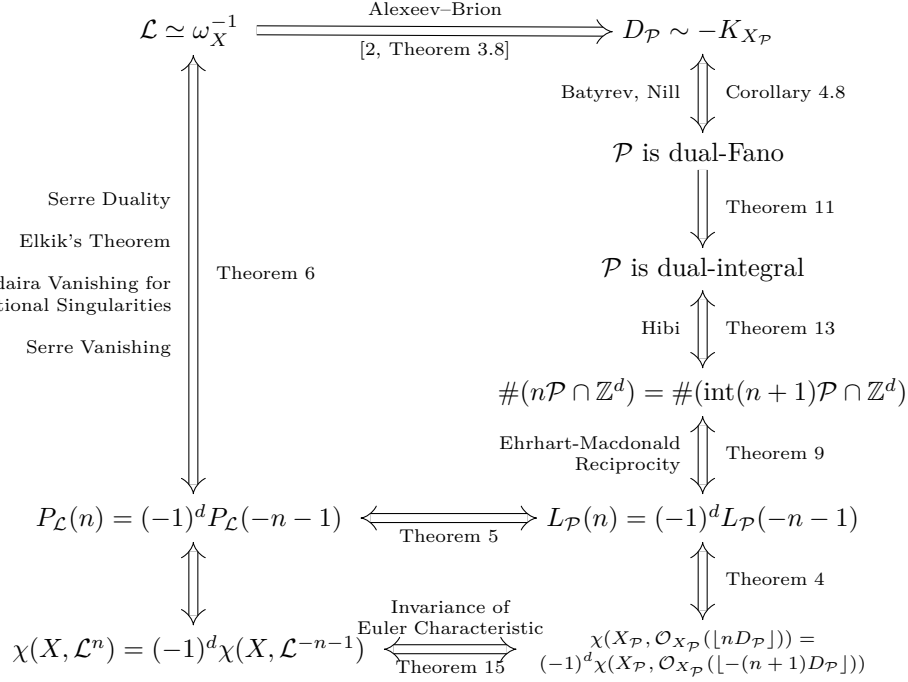
\begin{figure}[htb]
	\centering
	\caption{Sketch of the proof of \cref{thm:main}. Notation as in the formulation of its proof. The equations are supposed to hold for every $n\in\NN$.}\label{fig:proof}
	\begin{tikzpicture}[>=stealth,>=implies]
	\matrix (m) [matrix of nodes,align=center,row sep=3em,column sep=4em,minimum width=2em]
	{
		$\L\simeq\omega_X^{-1}$     & $D_\P \sim -K_{X_\P}$ \\
		& $\P$ is dual-Fano \\
		& $\P$ is dual-integral\\
		&$\#(n\P\cap\ZZ^d) = \#(\interior(n+1)\P \cap \ZZ^d)$\\
		$P_\L(n) = (-1)^d P_\L(-n-1)$ &$L_\P(n) = (-1)^d L_\P(-n-1)$\\		
		$\chi(X, \L^n) = (-1)^d\chi(X, \L^{-n-1})$ & {\scriptsize\begin{tabular}{c}$\chi(X_\P, \O_{X_\P}(\lfloor nD_\P\rfloor)) = $\\$(-1)^d \chi(X_\P, \O_{X_\P}(\lfloor -(n+1)D_\P\rfloor))$\end{tabular}}\\};
	\draw[->,double distance = 3pt] (m-1-1) -- node[above] {\scriptsize\begin{tabular}{c}Alexeev--Brion\end{tabular}} node[below] {\scriptsize\begin{tabular}{c}\cite[Theorem 3.8]{AB}\end{tabular}}(m-1-2);
	\draw[<->,double distance = 3pt] (m-1-1) -- node[left,align=right] {\scriptsize \begin{tabular}{r}Serre Duality\\\\ Elkik's Theorem\\\\Kodaira Vanishing for \\Rational Singularities\\\\Serre Vanishing\end{tabular}} node[right] {\scriptsize \begin{tabular}{l}\\\\\\\cref{thm:vanishing}\\\\\\\\\end{tabular}}(m-5-1);
	\draw[<->,double distance = 3pt] (m-5-1) -- (m-6-1);
	\draw[<->,double distance = 3pt] (m-6-1) -- node[above,align=center]{\scriptsize\begin{tabular}{c}Invariance of \\Euler Characteristic\end{tabular}} node[below] {\scriptsize\begin{tabular}{c}\cref{thm:twopolynomials}\end{tabular}} ([yshift=1.25pt]m-6-2.west);
	\draw[<->,double distance = 3pt] (m-5-1) -- node[below] {\scriptsize\begin{tabular}{c}\cref{cor:twopolynomials}\end{tabular}} (m-5-2);
	\draw[<->,double distance = 3pt] (m-5-2) -- node[left] {\scriptsize\begin{tabular}{r}Ehrhart-Macdonald\\Reciprocity\end{tabular}} node[right] {\scriptsize\begin{tabular}{l}\cref{thm:emr}\end{tabular}} (m-4-2);
	\draw[<->,double distance = 3pt] (m-6-2) --  node[right] {\scriptsize\begin{tabular}{l}\cref{thm:toricehrhart}\end{tabular}} (m-5-2);
	\draw[<->,double distance = 3pt] (m-4-2) -- node[left] {\scriptsize\begin{tabular}{r}Hibi\end{tabular}} node[right] {\scriptsize\begin{tabular}{l}\cref{thm:hibi}\end{tabular}}(m-3-2);
	\draw[<-,double distance = 3pt] (m-3-2) -- node[right] {\scriptsize\begin{tabular}{l}\cref{thm:dualfano}\end{tabular}} (m-2-2);
	\draw[<->,double distance = 3pt] (m-2-2) -- node[left] {\scriptsize\begin{tabular}{r}Batyrev, Nill\end{tabular}} node[right] {\scriptsize\begin{tabular}{l}\cref{cor:dualfano}\end{tabular}}(m-1-2);	
	\end{tikzpicture}
\end{figure}

The following consequence follows directly. It might prove useful in the hunt for mirror symmetry since it yields a necessary condition for reflexive polytopes to appear in the setting of toric degenerations.

\begin{theorem*}[\cref{thm:reflexive}]
	Let $X$ be a Gorenstein Fano variety and $\L$ an ample line bundle over $X$. Suppose that $(X, \L)$ admits a toric degeneration to the pair $(X_\P, D_\P)$ associated to the rational convex polytope $\P$. Then $\P$ is quasi-reflexive if and only if $\P$ is a lattice polytope and $\L$ is isomorphic to the anticanonical bundle over $X$.
\end{theorem*}

There is another, more delicate consequence. To find it, let us review the equivalences and implications of our proof schematically in \cref{fig:proof}. Essentially, we have proved that every dual-integral polytope appearing in the context of toric degenerations of polarized Gorenstein Fano varieties must already be a dual-Fano polytope. So the polytope must {\em know} that it is the limit of a polarized variety under a toric degeneration. We already noticed in \cref{thm:dpweil} that not every polytope can appear as the polytope associated to the toric limit divisor since that divisor must be Weil.

Philosophically speaking, this information should be contained in the combinatorics of the polytope. In fact, there is one on-the-fly result we obtained but never used in later arguments. In \cref{cor:degenerationpolynomial} we showed that every polytope whose associated toric pair is the limit of a polarized Gorenstein Fano variety must be a quasi-lattice polytope, i.e. its Ehrhart quasi-polynomial must be a polynomial.

By this reason\,---\,and from examples\,---\,we reach the following conjecture.

\begin{conj}\label{conj:dualfano}
	A convex polytope is dual-Fano if and only if it is a dual-integral quasi-lattice polytope.
\end{conj}

In other words this conjecture can be formulated as follows (by \cref{cor:dualfano}).

\begin{conj}
	Let $\P$ be a full-dimensional rational convex polytope and suppose for every ray $\rho$ in the normal fan $\Sigma_\P$ of $\P$ there exist primitive lattice vectors $u_\rho$ and a strictly positive integer $k_\rho > 0$ such that
	\[  \P = \{ x \in \RR^d \mid \langle x, u_\rho \rangle \leq k_\rho^{-1} \text{ for all } \rho\in\Sigma_\P(1)\}.\]
	Then the Ehrhart quasi-polynomial $L_\P$ of $\P$ is a polynomial if and only if all integers $k_\rho$, $\rho\in\Sigma_\P(1)$, are equal to $1$.
\end{conj}

This conjecture would immediately follow from the following one.

\begin{conj*}[\cref{conj:quasilattice}]
	The divisor associated to a full-dimensional rational convex polytope is a Weil divisor on the toric variety associated to the polytope if and only if the polytope is a quasi-lattice polytope.
\end{conj*}

Again, we can give a more combinatorial formulation of this conjecture.

\begin{conj}
	Let $\P$ be a full-dimensional rational convex polytope such that $0\in\interior\P$. Then the Ehrhart quasi-polynomial $L_\P$ of $\P$ is a polynomial if and only if there exist strictly positive integers $m_\rho > 0$ for every ray $\rho\in\Sigma_\P(1)$ such that
	\[  \P = \{ x \in \RR^d \mid \langle x, u_\rho \rangle \leq m_\rho \text{ for all } \rho\in\Sigma_\P(1)\}.\]
	Here $u_\rho$ denotes the primitive ray generator of the ray $\rho\in\Sigma_\P(1)$.
\end{conj}

\section{Recap: The Flag Variety Case}\label{sec:flag}

We will now briefly recall the special case of flag varieties. In the Main Theorem of \cite{S1} we essentially proved the following.

\begin{theorem}\label{thm:nobody}
	Let $G$ be a simple complex algebraic group, $P\subseteq G$ a parabolic and $\lambda$ a $P$-regular dominant integral weight. Let $\Delta(\lambda)$ be a Newton-Okounkov body associated to a full-rank valuation on $\bigoplus_{n\in\NN}H^0(G/P, \L_{n\lambda})$ whose valuation semigroup is finitely generated and saturated. Then the following are equivalent.
	\begin{enumerate}
		\item The weight $\lambda$ is the weight of the anticanonical line bundle over $G/P$.
		\item The Newton-Okounkov body $\Delta(\lambda)$ contains exactly one interior lattice point.
		\item The Newton-Okounkov body $\Delta(\lambda)$ is dual-integral.
	\end{enumerate}
\end{theorem}

However, we really would not have needed to restrict ourselves to Newton-Okounkov bodies. By \cref{rem:birational} these are just special cases of polytopes coming from toric degenerations of $(G/P, \L_\lambda)$. Hence the equivalence of claims (i) and (iii) in \cref{thm:nobody} is just a consequence of \cref{thm:main}. By \cref{cor:hibi} claim (iii) implies claim (ii). But the truly remarkable fact about flag varieties is that (ii) already implies (i) and (iii). 

The reason for this is that we have very precise control over the Ehrhart polynomial $L_{\Delta(\lambda)}$ because it is just given by Borel-Weil-Bott and Weyl's Dimension Formula. As stated in \cite[Lemma 3.1]{S1}, one can compute
\[ L_{\Delta(\lambda)}(n) = \prod_{\beta\in\Phi_P^+}\frac{\langle n\lambda+\rho, \beta^\vee\rangle}{\langle \rho, \beta^\vee\rangle} \]
where $\Phi_P^+$ denotes some subset of positive roots given by the choice of parabolic. Via some basic root combinatorics it is possible to verify that this polynomial fulfills Hibi's criterion (see \cref{rem:hibi}) if and only if it evaluates to $1$ at $n=-1$.

So the question remains whether a given partial flag variety admits a flat projective degeneration to a Gorenstein Fano toric variety (see \cref{question:simplified}). The answer is the following theorem that is essentially \cite[Theorem 6]{S3}.

\begin{theorem*}[\cref{thm:flag}]
	Let $G$ be a complex classical group or $\GG$. Then for any parabolic $P\subseteq G$ the partial flag variety $(G/P, \omega_{G/P}^{-1})$ admits a flat projective degeneration to an anticanonically polarized Gorenstein Fano toric variety.
\end{theorem*}

\begin{proof}
	By \cref{thm:main} any toric limit of $(G/P, \omega_{G/P}^{-1})$ will automatically be $\QQ$-Gorenstein Fano and anticanonically polarized. So we just have to show that there exists a toric limit $(X_\P, D_\P)$ with $\P$ being a lattice (hence reflexive) polytope because this means that $D_\P \sim -K_{X_\P}$ is Cartier.
	
	If $G$ is a complex classical group, we have shown in \cite[Theorem 5]{S3} that for every polarized partial flag variety $(G/P, \L_\lambda)$ there exists a string polytope that will be a lattice polytope if $\L_\lambda\simeq\omega_{G/P}^{-1}$. From \cite[Theorem 3.2]{AB} (or equivalently \cite[Theorem 1]{Ka} and \cite[Theorem 1, Corollary 2 and Lemma 3]{A}) it is known that $(G/P, \L_{\omega_{G/P}^{-1}})$ admits a toric degeneration (in our sense) to the toric variety associated to said string polytope, which must therefore be Gorenstein Fano.
		
	If $G$ is of type $\GG$ we can chose the polytope defined by Gornitskii in \cite{G}. One can easily calculate that it is always a lattice polytope. Additionally, by \cite[Section 8.2]{FaFL} and \cite[Theorem 6]{FaFL} it is clear that each polarized partial flag variety admits a toric degeneration to the polarized toric variety associated to the Gornitskii polytope, which must therefore be Gorenstein Fano.
\end{proof}

It would be very interesting to see whether \cref{thm:flag} also holds true for the exceptional Lie groups of type $\EE$, $\EEE$, $\EEEE$ and $\FF$. Additionally, it would be nice to see a class of polytopes that can be used for every type.

\end{document}